\documentclass[11pt,reqno]{amsart}

\usepackage[margin=1in]{geometry}
\usepackage[titletoc,title]{appendix}
\usepackage{amsthm}
\usepackage{fullpage}
\usepackage{dsfont}
\usepackage{comment}
\usepackage{graphicx}

\usepackage{hyperref}

\usepackage{todonotes}

\usepackage[useregional]{datetime2}

\usepackage{amsmath,amsfonts,amssymb,mathtools}

\usepackage{graphicx,float}

\newcommand{\Le}{\mathcal{L}_{\varepsilon}}
\newcommand{\Lel}{\mathcal{L}_{\varepsilon,\lambda}}

\newtheorem{theorem}{Theorem}
\newtheorem{lemma}[theorem]{Lemma}
\theoremstyle{definition}

\theoremstyle{remark}
\newtheorem*{remark}{Remark}

\title[]{Uniformly accurate  splitting schemes for  the Benjamin--Bona--Mahony equation with dispersive parameter}

\author{María Cabrera Calvo}
\address{LJLL (UMR 7598), Sorbonne Universit\'e, UPMC, 4 place Jussieu, 75005, Paris, France (M. Cabrera Calvo)}
\email{cabreracalvo@ljll.math.upmc.fr}

\author{Katharina Schratz}
\address{LJLL (UMR 7598), Sorbonne Universit\'e, UPMC, 4 place Jussieu, 75005, Paris, France (K. Schratz)}
\email{katharina.schratz@ljll.math.upmc.fr}

\begin{document}
\begin{abstract}
We propose a new class of uniformly accurate splitting methods for the Benjamin--Bona-Mahony equation which converge uniformly in the dispersive parameter $\varepsilon$. 
The proposed splitting schemes are furthermore   asymptotic convergent  and preserve the  KdV limit. We carry out a rigorous convergence analysis of the splitting schemes  exploiting the smoothing properties in the system. This will allow us to establish  improved error bounds with gain  either in regularity (for non smooth solutions) or in the dispersive parameter $\varepsilon$. The latter will be  interesting in regimes of a small dispersive parameter. We will in particular show that in the classical BBM case $P(\partial_x) = \partial_x$ our Lie splitting   does not require any spatial regularity, i.e, first order time convergence holds in $H^{r}$ for solutions in $H^{r}$  without any loss of derivative. This estimate holds  uniformly in $\varepsilon$. In regularizing regimes $\varepsilon=\mathcal{O}(1) $ we even gain a  derivative with our time discretisation at the cost of loosing in terms of $\frac{1}{\varepsilon}$.    Numerical experiments underline our theoretical findings.
\end{abstract}

\maketitle

\section{Introduction}
We consider the Benjamin--Bona-Mahony  (BBM) equation, 
\begin{gather}
\label{BBM}
\partial_t u(t,x) + \frac{P(\partial_x)}{1-\varepsilon \partial^2_x} u(t,x)+ \varepsilon \frac{\partial_x}{1-\varepsilon \partial^2_x} u^2(t,x)  =0,
\end{gather}
also known as the regularized long-wave equation, which
describes the behaviour of shallow water waves in direction 1+1. Here, $P$  denotes a polynomial in $\partial_x$ which we will define below. 

In recent years the theoretical and numerical analysis of the BBM equation has gained a lot of attention, see for instance \cite{avrin_generalized_1987,avrin_global_1985,fetecau_approximate_2005,stanislavova_global_2005}  for wellposedness results and \cite{avilez-valente_high-order_2009,dutykh_finite_2013,gucuyenen_strang_2017}  for numerical approximation methods, as well as  \cite{besse_artificial_2018} for the numerical analysis of the linearised BBM equation.  Due to their easy practical implementation and efficiency, splitting methods  provide   a particular attractive class of methods  to approximate the time dynamics of \eqref{BBM}. For an extensive overview on splitting  methods we refer to \cite{hairer_geometric_2006,McLacQ02} and the references therein, as well as to \cite{H1,H2} for their analysis in context of the Korteweg--de Vries equation. Previously proposed splitting methods for the BBM equation are   so far, however,   restricted to the smooth setting $\varepsilon =1$ (see, e.g., \cite{gucuyenen_strang_2017}), where in the classical case $P(\partial_x) = \partial_x$ the BBM equation \eqref{BBM} reduces to a \emph{regularized  ordinary differential equation}. The latter holds true due to the regularisation of the leading operator
\begin{equation}\label{regi}
\frac{P(\partial_x)}{1- \partial^2_x} = \mathcal{O}\left(\partial_x^{-1}\right) \quad \text{for}\quad P(\partial_x) = \partial_x.
\end{equation}
 Let us also mention  finite volume schemes for BBM that were  recently introduced in \cite{dutykh_finite_2013}  along  with  extensive numerical experiments.  A rigorous convergence analysis is  up to our knowledge, however, still  lacking in the general $\varepsilon$-dependent nonlinear setting  \eqref{BBM} which is strongly  governed by the dispersive parameter $0<\varepsilon\leq 1 $.  The main difficulty   lies in   regimes of small dispersion parameter $
\varepsilon \ll 1 
$, where the nice regularisation property \eqref{regi} breaks down   and  instead turns into a loss of derivative. This drastic change of behaviour holds true due to the estimate
\begin{equation*}\label{regi2}
\left\Vert \frac{\partial_x }{1- \varepsilon \partial^2_x} f \right\Vert_r \leq \text{min}\left\{ \frac{1}{\varepsilon} \Vert f \Vert_{r-1},\Vert f \Vert_{r+1}\right\} .
\end{equation*}

The aim of this paper lies in the development and  convergence analysis of high order  splitting schemes that reproduce the dynamics of the  solution  $u(t,x)$ of the generalised BBM equation \eqref{BBM} from the smooth setting $\varepsilon = 1 $ up to the {limit} regime $
\varepsilon \to 0 
$.  We  construct high order  splitting methods that  converge uniformly in $\varepsilon$, see also, Bao et al. for uniformly accurate schemes for Klein--Gordon type equations (e.g., \cite{Bao,BFS}). More precisely, we will prove low regularity error estimates in case of non smooth solutions with uniform convergence in $\varepsilon$, as well as improved error estimates for smooth solutions with gain in $\varepsilon$. The latter is in particular interesting in regimes of small dispersive parameter $\varepsilon$. Our main convergence result reads as follows: At order $p=1,2,3,4$  our splitting method of order $p$ satisfies the global error bound
\begin{gather}\label{result}
\|u(t_n)-u^n\|_r \leq \varepsilon^{1-\sigma} \tau^p c\left(\sup_{0\leq t \leq t_n}\|u(t)\|_{r+1+p\lambda -2\sigma}\right)\quad \text{with}\quad 0\leq \sigma \leq 1,
\end{gather}
where $\lambda$ denotes the degree of the leading differential operator $P(\partial_x)$. In the classical case $P(\partial_x)=\partial_x$ such that $\lambda = 1$ we in particular observe that no additional regularity of the solution is needed in our Lie splitting approximation (i.e., $p=1$), if choosing $\sigma =1$, at the cost of no longer gaining in terms of $\varepsilon$. Low regularity integrators   for other nonlinear dispersive equations such as  nonlinear Schr\"odinger and Korteweg--de Vries equations  were recently  introduced in \cite{ORS2,ORS1,RS1,RS2}.

Our convergence result \eqref{result} holds uniformly in  $\varepsilon$, see also  Figure \ref{EOC_plot}. In the regularizing regime $\varepsilon=\mathcal{O}(1) $ one can in addition show that thanks to the smoothing property \eqref{regi} we gain smoothness at the cost of loosing in terms of $\frac{1}{\varepsilon}$. More precisely, the following regularized convergence holds true for our Lie splitting method
\begin{gather}\label{resultSmooth}
\|u(t_n)-u^n\|_r \leq \tau \frac{1}{\varepsilon} c\left(\sup_{0\leq t \leq t_n}\|u(t)\|_{r-2+\lambda}\right) .
\end{gather}
This result is, in particular in the classical setting $P(\partial_x)=\partial_x$ (that is $\lambda = 1$),  interesting from a theoretical point of you as we gain in regularity with our Lie discretisation. Note that in the latter case first order time convergence holds in $H^r$ for solutions in $H^{r-1}$. However, in practical computations  one needs to couple the Lie time discretisation with a suitable spatial discretisation which will again require some smoothness of the initial data.

Our uniformly accurate splitting schemes for BBM furthermore allow us to capture the limit regime 
where  the BBM equation in the classical case $ P(\partial_x) = \partial_x$  collapses to the Korteweg--de Vries (KdV) equation (see, e.g.,  \cite{besse_artificial_2018})
\begin{equation}\label{kdvapp}
\partial_t u_{\tiny\text{KdV}} + \partial_x u_{\tiny\text{KdV}} +\varepsilon  \partial_x u_{\tiny\text{KdV}}^2 + \varepsilon \partial_x^3 u_{\tiny\text{KdV}} =0.
\end{equation}

In the error analysis of the splitting schemes we will  heavily exploit the structure of the operator
\begin{equation}\label{Le}
 \mathcal{L}_{\varepsilon} = \frac{  \partial_x}{1-\varepsilon \partial^2_x}
\end{equation}
and its  smoothing property which strongly depends   on the dispersive parameter $\varepsilon$. Note that for $\varepsilon =\mathcal{O}(1)$  we gain  regularization thanks to the observation that $ \mathcal{L}_{\varepsilon}=\mathcal{O}\left(\frac{1}{\varepsilon}\partial_x^{-1}\right) $, whereas for $\varepsilon \to 0 $ we loose a full derivative due to the limit behaviour $ \mathcal{L}_{\varepsilon=0}=\mathcal{O}\left(\partial_x\right) $. Interpolating this gain in regularity and loss in $\varepsilon$   will allow us to establish the improved global error estimates \eqref{result}.  We will focus on methods up to order four. However, our construction and analysis can be extended to arbitrary high order.


\noindent{\bf Outline of the paper.}
In Section \ref{sec:main} we introduce the general framework. Then we will discuss the Lie, Strang and higher order splitting method in Sections \ref{sec:Lie} to Section \ref{sec:four}. In each section we will develop a uniformly accurate splitting scheme up to the desired order and carry out its global  error analysis. We will prove in each section the global error estimate \eqref{result} for the particular value of   $1\leq p \leq 4$. Numerical experiments in Section \ref{sec:num} underline our theoretical findings.

\noindent{\bf Notation and assumptions.}
 For practical implementation issues we will impose periodic boundary conditions that is  $x \in \mathbb{T} =[-\pi ,\pi]$ and assume that  $P$ is a polynomial of degree $\lambda\geq 1$ such that $\mathrm{Re}\,P(ix)= 0$, for all $x\in\mathbb{R}$. We will fix $r>\frac{1}{2}$ and further set for $\varepsilon\geq 0$
\begin{gather*}
\mathcal{L}_{\varepsilon,\lambda} = \frac{P(\partial_x)}{1-\varepsilon \partial^2_x}.
\end{gather*}
We denote by $\|.\|_r$ the standard $H^r(\mathbb{T})$ norm, where in particular, for this choice of $r$, the standard bilinear estimate
\begin{gather}
\label{bil_est}
\|fg\|_r \leq C_r \|f\|_r \|g\|_r
\end{gather}
holds for all $f,g\in H^r(\mathbb{T})$ and some constant $C_r>0$.

\section{General splitting framework}\label{sec:main}
In this section we present the general framework of this paper. The key idea is the following: instead of solving the full problem \eqref{BBM}, we split the BBM equation \eqref{BBM} into the linear and nonlinear subproblem
\begin{align}
\partial_t w(t,x) &= - \varepsilon \Le w^2(t,x), \tag{S1}\label{S1} \\
\partial_t v(t,x) &= - \Lel v(t,x). \tag{S2}\label{S2}
\end{align}
To obtain an approximation to the original solution $u$ of \eqref{BBM} we then  compose  the solutions of the corresponding subproblems for a small time step size $\tau$ up to the desired order.

On the one hand, we note that the linear subproblem \eqref{S2} can be solved exactly   in   Fourier space with the exact solution  $v(t)=e^{-t\Lel}v(0)$. Indeed, we see that 
\begin{align*}
\partial_t v(t,x)=\partial_t  \sum_{k\in\mathbb{Z}}e^{ik\cdot x} \hat{v}_k(t)= \Lel \sum_{k\in\mathbb{Z}} e^{ik\cdot x} \hat{v}_k(t) = \sum_{k\in\mathbb{Z}} \frac{P(ik)}{1+\varepsilon k^2} e^{ik\cdot x} \hat{v}_k(t)
\end{align*}
such that by comparison in the Fourier basis we obtain
$$
\hat{v}_k(t) = e^{t\frac{P(ik)}{1+\varepsilon k^2} }\hat{v}_k(0).
$$

  The nonlinear subproblem \eqref{S1}, on the other hand, can not be solved exactly.  Thus we will have to approximate it numerically. For this purpose we consider the  corresponding  mild solution
\begin{gather}
\label{duhamelW}
w(t_n+\tau)=w(t_n)-\varepsilon\Le \int_0^{\tau} w^2 (t_n+s) \,\mathrm{d}s.
\end{gather}
Then we find approximations $w^n\approx w(t_n,x)$ by means of truncated Taylor series expansions of $w^2(t_n+s)$. In order to guarantee  the full order of convergence of the splitting schemes we have to  use a (high order) expansion  of the correct order for the approximation of the nonlinear problem \eqref{S1}. The detailed approximation for the Lie, Strang and higher order splitting methods is given below.

The above considerations will allow us to construct compositions of the form
\begin{align}
\label{general_splitting_scheme}
u(t_{n+1},x)\approx u^{n+1}=\Psi^{\tau}(u^n)= \big(\Psi_2^{d_N\tau}\circ \Psi_1^{c_N\tau} \circ \dots \circ \Psi_2^{d_2\tau}\circ \Psi_1^{c_2\tau} \circ \Psi_2^{d_1\tau}\circ \Psi_1^{c_1\tau}\big)(u^n),
\end{align}
where  $\Psi_2^{\tau}$ will denote the exact flow of the linear subproblem \eqref{S2} and  $\Psi_1^{\tau}$  the numerical flow given by a suitable approximation to  \eqref{duhamelW}.  The  real valued coefficients $c_1,\,d_1,\,\dots,c_N,\,d_N$, $N\geq 1$  are chosen according to order conditions for splitting methods, see, e.g.,  \cite{hairer_geometric_2006}. We will give the detailed construction up to order four in more detail below.

Let us first mention an important lemma on the smoothing property of the operator $\Le$ which will allow  weaker regularity assumptions in the corresponding local error bounds for non smooth solutions and gain in $\varepsilon$ for small dispersive parameters.
\begin{lemma}
\label{lemma_bound_eLe}
Let $f\in H^{r+1-2\sigma}(\mathbb{T})$. It holds that
\begin{gather*}
\|\varepsilon\Le f\|_r \leq \varepsilon^{1-\sigma} \|f\|_{r+1-2\sigma}, \quad 0\leq \sigma \leq 1.
\end{gather*}
\end{lemma}
\begin{proof}
We have that
\begin{gather*}
\|\varepsilon\Le f\|_r^2 =  \sum_{k\in\mathbb{Z}} (1 + |k|)^{2r} \bigg| \frac{\varepsilon ik}{1+\varepsilon k^2} \bigg|^2 |\hat{f}_k|^2
\leq \sum_{k\in\mathbb{Z}} (1 + |k|)^{2r} \bigg|\frac{\varepsilon k }{(\varepsilon k^2)^\sigma}\bigg|^2|\hat{f}_k|^2 \leq \varepsilon^{2(1-\sigma)} \|f\|_{r+1-2\sigma}^2.
\end{gather*}
\end{proof}
Furthermore,  we will encounter partial derivatives with respect to time of the nonlinearity $w^2(t)$, as they naturally appear in the remainder terms of the truncated Taylor series expansion. Thus,  we collect the regularity result in  the following lemma.
\begin{lemma}
\label{bound_partialw}
Let $w(t)\in H^r(\mathbb{T})$ be the solution of \eqref{S1}. Then for all $0\leq t \leq T$ we have 
$$\|\partial_t^m w^2(t)\|_r\leq  \varepsilon^{1-\sigma} K(\sup_{0\leq t\leq T}\|w(t)\|_{\text{max}{(r+\lambda-2\sigma,r)}}).
$$
\end{lemma}
\begin{proof}
The claim follows by  induction together with Lemma \ref{lemma_bound_eLe} and the bilinear estimate~\eqref{bil_est}. 
\end{proof}

Before we begin, let us address the generalisation of polynomials $P(\partial_x)$ in the following remark.
\begin{remark}
In higher order splitting methods \eqref{general_splitting_scheme}
 we (in general) encounter negative coefficients. Those are proven necessary in \cite{blanes_necessity_2005}, and in particular, it is shown that the threshold is order three. This explains the additional assumptions on  $P$ which guarantees that all weighted operator flows  $\Psi_2^{d_j \tau}$  are linear isometries such that negative coefficients in the high order splitting methods can be used (as we can go both forward and backward in time). An alternative approach of higher order splitting methods is given in \cite{hansen_high_2009}, via the use of complex coefficients. This would  allow for a more general structure of $P$.
\end{remark}

In the following section we will collect some important estimates in the stability and local error analysis of the splitting methods.
\subsection{Some important estimates}\label{sec:estis}
For $u(t)\in H^r(\mathbb{T})$, $t\in\mathbb{R}$, we define the exact flow of the BBM equation \eqref{BBM} by
\begin{gather}
\label{exact_flow}
\Phi^{\tau}(u(t_n)) := e^{-\tau \Lel} u(t_n) - \varepsilon \Le \int_0^{\tau} e^{-(\tau-s)\Lel}u^2(t_n+s)\,\mathrm{d}s
\end{gather}
such that in particular we have $u(t_n+\tau)=\Phi^{\tau}(u(t_n))$.

In order to carry out the error analysis we will make use of the following lemma that can be proven in the same fashion as Lemma \ref{lemma_bound_eLe}.
\begin{lemma}
\label{lemma_bound_Le}
Let $f\in H^{r+\lambda}$. It holds that
\begin{gather*}
\| \Le f \|_r \leq \min \bigg\{ \frac{1}{\varepsilon}\|f\|_{r-1}, \|f\|_{r+1} \bigg\}\quad \text{and} \quad
\| \Lel f \|_r \leq \min \bigg\{ \frac{1}{\varepsilon}\|f\|_{r+\lambda-2}, \|f\|_{r+\lambda} \bigg\}.
\end{gather*} 
\end{lemma}
\begin{proof}
For $f\in H^{r+\lambda}(\mathbb{T})$ we have
\begin{gather*}
\|\Le f\|_r^2 =  \sum_{k\in\mathbb{Z}} (1 + |k|)^{2r} \bigg| \frac{ik}{1+\varepsilon k^2} \bigg|^2 |\hat{f}_k|^2
\leq \sum_{k\in\mathbb{Z}} (1 + |k|)^{2r} \bigg|\frac{1}{\varepsilon k}\bigg|^2|\hat{f}_k|^2 
\end{gather*}
and, on the other hand,
\begin{gather*}
\|\Lel f\|_r^2 =  \sum_{k\in\mathbb{Z}} (1 + |k|)^{2r} \bigg| \frac{P(ik)}{1+\varepsilon k^2} \bigg|^2 |\hat{f}_k|^2
\leq \sum_{k\in\mathbb{Z}} (1 + |k|)^{2r} \bigg|\frac{P(ik)}{\varepsilon k^2}\bigg|^2|\hat{f}_k|^2 \leq \sum_{k\in\mathbb{Z}} (1 + |k|)^{2r} \bigg|\frac{G(ik)}{\varepsilon }\bigg|^2|\hat{f}_k|^2 ,
\end{gather*}
for some polynomial $G$ of degree $\lambda-2$.
\end{proof}
The following lemmata will be essential in the stability analysis of the splitting methods.
\begin{lemma}
\label{lemma_lin_iso}
For all $f\in H^r(\mathbb{T})$ and all $t\in\mathbb{R}$ it holds $\|e^{t\Lel}f\|_r = \|f\|_r$.
\end{lemma}
\begin{proof}
Let $f\in H^r(\mathbb{T})$ and $t\in\mathbb{R}$. Then we have that
\begin{gather}
\label{lin_iso}
\| e^{t\Lel} f\|_r^2 = \sum_{k\in\mathbb{Z}} (1 + |k|)^{2r} \big|e^{\frac{P(ik)t}{1+\varepsilon k^2}}\big|^2 |\hat{f}_k|^2 
= \sum_{k\in\mathbb{Z}} (1 + |k|)^{2r} |\hat{f}_k|^2 
= \|f\|_r^2.
\end{gather}
\end{proof}
\begin{lemma}
\label{lemma_e-1bound}
For all $f\in H^r(\mathbb{T})$ and all $t\in\mathbb{R}$ it holds $\|(e^{t\Lel}-1)f\|_r = |t|\|f\|_{r+\lambda}$.
\end{lemma}
\begin{proof}
\begin{align*}
\|(e^{t\Lel}-1)f\|_{r}^2 &= \sum_{k\in\mathbb{Z}} (1 + |k|)^{2r} \big|e^{\frac{P(ik)t}{1+\varepsilon k^2}}-1 \big|^2 |\hat{f}_k|^2 =\sum_{k\in\mathbb{Z}} (1 + |k|)^{2r} \bigg| P(ik)t \frac{e^{\frac{P(ik)t}{1+\varepsilon k^2}}-1}{P(ik)t} \bigg|^2 |\hat{f}_k|^2\\
&\leq \sum_{k\in\mathbb{Z}} (1 + |k|)^{2r} | P(ik)t |^2 |\hat{f}_k|^ =|t|^2 \|f\|_{r+\lambda}^2.
\end{align*}
\end{proof}

\section{Lie Splitting for BBM}\label{sec:Lie}
\label{liesplitting}
In this section we  construct a uniformly accurate Lie splitting method for the BBM equation \eqref{BBM} and carry out its error analysis.  
\subsection{Derivation}
To develop the first order Lie splitting method we have to derive a locally second-order approximation to the nonlinear supbroblem \eqref{S1}.  Taylor series expansion of $w^2(t_n+s)$ gives
\begin{align*}
w(t_n+\tau) &= w(t_n) - \varepsilon\Le \int_0^{\tau} w^2(t_n+s)\,\mathrm{d}s\\
&=w(t_n) - \tau \varepsilon\Le  w^2(t_n) + R_1(w),
\end{align*}
where by  Lemma \ref{lemma_bound_eLe} the remainder $R_1(w)$ satisfies the bound \begin{gather*}
\|R_1(w)\|_r = \bigg\| \varepsilon\Le \int_0^{\tau} s\,\mathrm{d}s\,\partial_t w^2(t)_{t=\xi} \bigg\|_r \leq  \varepsilon^{2(1-\sigma)} \tau^2 K,
\end{gather*}
where $\xi\in [t_n,t_{n+1}]$, and some constant $K=K(\sup_{0\leq t\leq T}\|w(t)\|_{\text{max}{(r+\lambda-2\sigma,r)}})>0$. This motivates us to define the numerical flow as follows
\begin{align}
\label{Psi2}
\Psi_1^{\tau}(w(t_n)) := w(t_n) - \tau \varepsilon\Le  w^2(t_n), \quad
\Psi_2^{\tau}(v(t_n)) := e^{-\tau \Lel}v(t_n)
\end{align}
such that the corresponding Lie splitting $u^{n+1} = \Psi_2^{\tau} \big( \Psi_1^{\tau}(u^n) \big)$ takes the form 
\begin{align}
\label{MET1}
&u^{n+1} = e^{-\tau \Lel} \big(u^n - \tau \varepsilon\Le  (u^n)^2  \big),\\
&u^0 = u(0). \notag
\end{align}
We will prove below that the Lie splitting $u^n$ defined in  \eqref{MET1} approximates the exact BBM solution $u(t)$ at time $t_n$ up to order one.

\subsection{Error Analysis of the Lie splitting scheme}
The estimates in Section \ref{sec:estis} allow us to prove the following convergence result.
\begin{theorem}\label{thmLie}
Fix $0\leq \sigma \leq 1$  and $r\geq 0$ such that $r+1-2\sigma+\lambda>1/2$ and assume that the exact solution of \eqref{BBM} satisfies $u \in \mathcal{C}\left([0,T],H^{r+1-2\sigma+\lambda}(\mathbb{T})\right)$. Then there exists a $\tau_0>0$ such that for all $\tau\leq\tau_0$ and $t_n\leq T$ the Lie splitting $u^n$ defined in \eqref{MET1} satisfies the global error estimate
\begin{gather*}
\|u(t_n)-u^n\|_r \leq \tau \varepsilon^{1-\sigma} K, 
\end{gather*}
for a constant $K=K(\sup_{t\in[0,T]}\|u(t)\|_{r+1-2\sigma+\lambda})>0$.
\end{theorem} 
\begin{proof}
We treat the local error and the stability estimates separately. Then we can conclude the proof via a Lady Windemere's fan argument (see, e.g.,  \cite{hairer_geometric_2006}).\\
{\bf Local error analysis.} The local error consists of two parts: We need to consider both the error that arises from the splitting ansatz, known as the commutator error, and the error of the truncated Taylor series expansion within the approximation of the nonlinear problem \eqref{S1}. Adding and subtracting the term $\varepsilon\Le \int_0^{\tau}e^{-(\tau-s)\Lel}u^2(t_n)\,\mathrm{d}s$ gives
\begin{equation}\label{Ki1}
\begin{aligned}
\big\| \Phi^{\tau}(u(t_n)) - \Psi_2^{\tau} \big( \Psi_1^{\tau}(u(t_n)) \big) \big\|_r \leq \,&\bigg\| \varepsilon\Le e^{-\tau\Lel} \int_0^{\tau} e^{s\Lel} (u^2(t_n+s) - u^2(t_n)) \,\mathrm{d}s  \bigg\|_r\\
&+ \bigg\| \varepsilon\Le \int_0^{\tau} (e^{s\Lel}-1)u^2(t_n)\,\mathrm{d}s  \bigg\|_r\\
\leq \,&
\varepsilon^{1-\sigma} \bigg\|  \int_0^{\tau} e^{s\Lel} (u^2(t_n+s) - u^2(t_n)) \,\mathrm{d}s  \bigg\|_{r+1-2\sigma}\\ &+ \varepsilon^{1-\sigma} \bigg\|\int_0^{\tau} (e^{s\Lel}-1)u^2(t_n)\,\mathrm{d}s  \bigg\|_{r+1-2\sigma}.
\end{aligned}
\end{equation}
Here, the second inequality follows by Lemma \ref{lemma_bound_eLe}. We may bound the first term using the following observation
\begin{align*}
u^2(t_n+s) - u^2(t_n) = s \,\partial_t u^2(t)_{t=\xi},
\end{align*}
for some $\xi\in[t_n,t_{n+1}]$. Thus, we obtain thanks to Lemma \ref{lemma_lin_iso} that
\begin{align*}
\bigg\|  \int_0^{\tau} e^{s\Lel} (u^2(t_n+s) - u^2(t_n)) \,\mathrm{d}s  \bigg\|_{r+1-2\sigma} \leq \tau^2 K,
\end{align*}
for some $K=K(\sup_{t\in[0,T]} \|u(t)\|_{r+1-2\sigma+\lambda})>0$.  As for the second term in \eqref{Ki1}  it follows by Lemma~\ref{lemma_e-1bound} that
\begin{align*}
\bigg\|\int_0^{\tau} (e^{s\Lel}-1)u^2(t_n)\,\mathrm{d}s  \bigg\|_{r+1-2\sigma} \leq \tau^2 K
\end{align*}
for some $K=K(\|u(t_n)\|_{r+1-2\sigma+\lambda})>0$. Plugging the above two estimates into \eqref{Ki1}, we obtain the full local error bound
\begin{align}\label{locLie}
\big\| \Phi^{\tau}(u(t_n)) - \Psi_2^{\tau} \big( \Psi_1^{\tau}(u(t_n)) \big) \big\|_r \leq \tau^2 \varepsilon^{1-\sigma} K_1(\sup_{\xi\in[0,T]} \|u(\xi)\|_{r+1-2\sigma+\lambda}).
\end{align}
\\
{\bf Stability analysis.} Using \eqref{bil_est}, Lemma \ref{lemma_bound_eLe}, and Lemma \ref{lemma_lin_iso} we obtain the following stability estimate
\begin{equation}\label{stabLie}
\begin{aligned}
\big\| \Psi_2^{\tau} \big( \Psi_1^{\tau}(f(t_n)) \big) - \Psi_2^{\tau} \big( \Psi_1^{\tau}(g(t_n)) \big) \big\|_r &\leq \| f(t_n)-g(t_n)\|_r +  \| \tau\varepsilon \Le e^{-\tau\Lel}(f^2(t_n)-g^2(t_n))\|_r\\
 &\leq \| f(t_n)-g(t_n)\|_r +  \tau C_r \|f(t_n)+g(t_n)\|_{r-1} \| f(t_n)-g(t_n)\|_{r-1}  \\
&\leq (1+\tau M_1)\|f(t_n)-g(t_n)\|_r,
\end{aligned}
\end{equation}
where $M_1=M_1(r,\|f(t_n)\|_{r-1},\|g(t_n)\|_{r-1})>0$. \\
{\bf Global error analysis.} 
Inserting zero in terms of $ \Psi_2^{\tau} \big( \Psi_1^{\tau}(u(t_n)) \big)$ we obtain thanks to the triangle inequality that
\begin{align*}
\Vert u(t_{n+1}) - u^{n+1} \Vert_r &  = \Vert \Phi^\tau(u(t_n)) -  \Psi_2^{\tau} \big( \Psi_1^{\tau}(u^n) \big) \Vert_r\\ & \leq
  \Vert \Phi^\tau(u(t_n)) -  \Psi_2^{\tau} \big( \Psi_1^{\tau}(u(t_n)) \big) \Vert_r +  \Vert  \Psi_2^{\tau} \big( \Psi_1^{\tau}(u(t_n)) \big)  -  \Psi_2^{\tau} \big( \Psi_1^{\tau}(u^n) \big) \Vert_r .
\end{align*}
The local error estimate \eqref{locLie} together with the stability bound \eqref{stabLie} allows us to conclude 
\begin{align*}
\Vert u(t_{n+1}) - u^{n+1} \Vert_r &   \leq
\tau^2 \varepsilon^{1-\sigma} K_1(\sup_{\xi\in[0,T]} \|u(\xi)\|_{r+1-2\sigma+\lambda}) + (1+\tau M_1) \Vert u(t_n) - u^n\Vert_r  .
\end{align*}
where $M_1=M_1(r,\|u(t_n)\|_{r-1},\|u^n\|_{r-1})>0$.
Thus the global first order convergence follows by a Lady Windermere's fan argument (see \cite{hairer_geometric_2006}). 
\end{proof}
\begin{remark}
Note that in our stability argument \eqref{stabLie} we can not exploit any gain in $\varepsilon$ as we can not measure the right hand side in a stronger norm than  the left hand side. The latter would cause the break down of our stability argument.\end{remark}

\begin{remark}
The regularized convergence estimate \eqref{resultSmooth} follows similarly by observing that the local Lie splitting error  satisfies thanks to  Lemma \ref{lemma_bound_Le} the regularized error estimate
\begin{align*}
\big\| \Phi^{\tau}(u(t_n)) - \Psi_2^{\tau} \big( \Psi_1^{\tau}(u(t_n)) \big) \big\|_r \leq \tau^2  \frac{1}{\varepsilon} K(\sup_{\xi\in[0,T]} \|u(\xi)\|_{r-2+\lambda})
\end{align*}
due to the smoothing property of the operator $\Lel$ (cf. also \eqref{regi}). In order to close the stability argument in the case that $-2+\lambda< 0$ one thereby first needs to prove convergence in $H^{r-2+\lambda}$ for solutions in $H^{r-2+\lambda}$. This will allow us to conclude a priori  the boundedness of the numerical solution in $H^{r-2+\lambda}$. 
\end{remark}

\section{Strang Splitting for BBM}
In this section we  construct a uniformly accurate Strang  splitting method for the BBM equation~\eqref{BBM} and carry out its error analysis. 
\subsection{Derivation}
We look at the subproblems \eqref{S1}, \eqref{S2} and define the Strang splitting
\begin{gather}
\label{MET2}
\hat{\Psi}^{\tau}:= \Psi_2^{\frac{\tau}{2}}\circ \hat{\Psi}_1^{\tau} \circ \Psi_2^{\frac{\tau}{2}}.
\end{gather}
We have to derive a suitable numerical $\hat{\Psi}_1$ flow. Thereby it is essential to develop a second (or higher) order approximation to the nonlinear subproblem \eqref{S1} as otherwise the full second order convergence of the Strang splitting \eqref{MET2} would break down.  Using Taylor series expansion of $w^2(t_n+s)$ gives
\begin{align*}
w(t_n+\tau)&=w(t_n)-\varepsilon\Le \int_0^{\tau} w^2 (t_n+s) \,\mathrm{d}s\\
&=w(t_n)-\varepsilon\Le \int_0^{\tau} \big( w(t_n) + 2s w(t_n) \partial_t w(t)_{t=t_n} \big) \,\mathrm{d}s + R_2(w)\\
&= w(t_n)- \tau \varepsilon\Le w(t_n) + \tau^2\varepsilon^2 \Le w(t_n)\Le w^2 (t_n) + R_2(w),
\end{align*}
where $R_2(w)$ is the remainder of the truncated Taylor series expansion and thus satisfies, for some $\xi\in[t_n,t_{n+1}]$,
\begin{align*}
\|R_2(w)\|_r = \bigg\| \varepsilon\Le\int_0^{\tau} \frac{s^2}{2} \partial^2_{t} w^2(t)_{t=\xi}\,\mathrm{d}s \bigg\|_{r} \leq \frac{\tau^3}{6} \sup_{t_n \leq t \leq t_{n+1}} \|\partial^2_{t} w^2(t)\|_{r}\leq  \varepsilon^{1-\sigma} \tau^3K,
\end{align*}
for some $K=K(\sup_{t_n \leq t \leq t_{n+1}}\|w(t)\|_{r+1-2\sigma})>0$. The structure of the error constant  follows by Lemma \ref{bound_partialw}.

The above expansion motivates us to   define 
\begin{align}\label{Psi1Strang}
\hat{\Psi}_1^{\tau}(w^n):= w^n- \tau \varepsilon\Le w^n + \tau^2\varepsilon^2 \Le w^n\Le (w^n)^2
\end{align}
such that the Strang splitting \eqref{MET2} takes the form
\begin{gather}
\label{strang_comp1}
 \hat{\Psi}^{\tau}(u^n) = e^{-\tau\Lel}u^n-\tau\varepsilon e^{-\frac{\tau}{2}\Lel} \Le \big( e^{-\frac{\tau}{2}\Lel} u^n \big)^2 + \tau^2\varepsilon^2 e^{-\frac{\tau}{2}\Lel} \Le \big(e^{-\frac{\tau}{2}\Lel} u^n\big) \Le \big( e^{-\frac{\tau}{2}\Lel} u^n \big)^2.
\end{gather}
\subsection{Error Analysis of the Strang splitting scheme}
\begin{theorem}
Fix $0\leq \sigma \leq 1$  and $r\geq 0$ such that $r+1-2\sigma+2\lambda>1/2$  and assume that the exact solution of \eqref{BBM} satisfies $u \in \mathcal{C}\left([0,T],H^{r+1+2\lambda-2\sigma}(\mathbb{T})\right)$. Then there exists a  $\tau_0>0$ such that for all $0<\tau\leq\tau_0$ and $t_n\leq T$ the Strang splitting  $u^n$ defined in  \eqref{strang_comp1} satisfies the global error estimate
\begin{gather*}
\|u(t_n)-u^n\|_r \leq \tau^2 \varepsilon^{1-\sigma} K, 
\end{gather*}
for a constant $K=K(\sup_{t\in [0,T]}\|u(t)\|_{r+1+2\lambda-2\sigma })>0$.
\end{theorem}
\begin{proof}
We employ the same technique as in the previous section and treat local error and stability estimates separately.

\noindent {\bf Local Error analysis.}
Recall the structure of the Strang splitting scheme
\eqref{strang_comp1}. On the other hand, if we apply the midpoint rule to the exact flow we obtain
\begin{align*}
\Phi^{\tau}(u(t_n))=\,& e^{-\tau \Lel}u(t_n) - \varepsilon\Le \int_0^{\tau} e^{-(\tau-s)\Lel} \Phi^s(u(t_n))^2 \,\mathrm{d}s \\
=\,& e^{-\tau \Lel}u(t_n) - \tau\varepsilon\Le e^{-\frac{\tau}{2}\Lel}\Phi^{\frac{\tau}{2}}(u(t_n))^2  + R'\\
=\,&e^{-\tau \Lel}u(t_n) - \tau\varepsilon\Le e^{-\frac{\tau}{2}\Lel}\big(e^{-\frac{\tau}{2}\Lel} u(t_n) \big)^2 \\
&+ 2 \tau\varepsilon^2\Le e^{-\frac{\tau}{2}\Lel} \big(e^{-\frac{\tau}{2}\Lel} u(t_n)\big)\Le \int_0^{\frac{\tau}{2}}e^{-(\frac{\tau}{2}-\xi)\Lel}\Phi^{\xi}(u(t_n))^2\,\mathrm{d}\xi+R'',
\end{align*}
where $R'$ is the remainder of the midpoint rule, and thus has the form
\begin{align}
\label{strang_err1}
\|R'\|_r\leq \bigg\|\varepsilon\Le\int_0^{\tau} \frac{s^2}{2} \partial_s^2 \Phi^s(u(t_n))^2   \bigg\|_r  \leq \tau^3 \varepsilon^{1-\sigma} K',
\end{align}
for some $K'=K'(\|u\|_{r+2\lambda-2\sigma+1})>0$. $R''$ is the remainder that consists of $R'$ as well as the integral term of $\Phi^{\frac{\tau}{2}}(u(t_n))$ squared, for which we have by \eqref{bil_est} and Lemmas \ref{lemma_bound_eLe} and \ref{lemma_lin_iso}  that 
\begin{align}
\label{strang_err3}
\|R''\|_r\leq \|R'\|_r + \bigg\| \tau \varepsilon\Le e^{-\frac{\tau}{2}\Lel}\bigg(\varepsilon\Le \int_0^{\frac{\tau}{2}}e^{-(\frac{\tau}{2}-\xi)\Lel}\Phi^{\xi}(u(t_n))^2\,\mathrm{d}\xi \bigg)^2  \bigg\|_r \leq \tau^3 \varepsilon^{1-\sigma} K'',
\end{align}
for some $K''=K''(\|u\|_{r+1-2\sigma+2\lambda})>0$. Next we apply   the approximations $$e^{-(\frac{\tau}{2}-\xi)\Lel}=1+\mathcal{O}(\xi\Lel), \qquad \Phi^{\xi}(u(t_n))^2=u^2(t_n)+\mathcal{O}(\xi\Le)$$ in the expansion of the exact solution which gives
\begin{equation}
\begin{aligned}
\label{strang_comp2}
\Phi^{\tau}(u(t_n))& = e^{-\tau \Lel}u(t_n) - \tau\varepsilon\Le e^{-\frac{\tau}{2}\Lel}\big(e^{-\frac{\tau}{2}\Lel} u(t_n) \big)^2 +  \tau^2\varepsilon^2\Le e^{-\frac{\tau}{2}\Lel} \big(e^{-\frac{\tau}{2}\Lel} u(t_n)\big)\Le u^2(t_n) \\& + R''',
\end{aligned}
\end{equation}
where the remainder $R'''$ satisfies
\begin{equation}
\begin{aligned}
\label{strang_err2}
\|R'''\|_r \leq\,& \|R''\|_r + \bigg\| 2 \tau\varepsilon^2\Le e^{-\frac{\tau}{2}\Lel} \big(e^{-\frac{\tau}{2}\Lel} u(t_n)\big)\Le \int_0^{\frac{\tau}{2}}(e^{-(\frac{\tau}{2}-\xi)\Lel}-1)\Phi^{\xi}(u(t_n))^2\,\mathrm{d}\xi \bigg\|_r \\
&+ \bigg\| 2 \tau\varepsilon^2\Le e^{-\frac{\tau}{2}\Lel} \big(e^{-\frac{\tau}{2}\Lel} u(t_n)\big)\Le \int_0^{\frac{\tau}{2}}(\Phi^{\xi}(u(t_n))^2-u^2(t_n))\,\mathrm{d}\xi \bigg\|_r\\
\leq\,&\tau^3 \varepsilon^{1-\sigma} K''',
\end{aligned}
\end{equation}
for some $K'''=K'''(\|u\|_{r+1-2\sigma+2\lambda})>0$. The second inequality in \eqref{strang_err2} follows by \eqref{bil_est} and Lemmas \ref{lemma_bound_eLe}, \ref{lemma_bound_Le}, \ref{lemma_lin_iso} and \ref{lemma_e-1bound}.

Comparing \eqref{strang_comp1} and \eqref{strang_comp2}, we see that the terms of order 0 and 1 in $\tau$ coincide, and those of order~3 and higher are collected in the remainder term $R'''$. For the terms of order 2 we see that
\begin{equation}
\label{strang_err3}
\begin{aligned}
\| \varepsilon^2 e^{-\frac{\tau}{2}\Lel} \Le \big(e^{-\frac{\tau}{2}\Lel} u(t_n)\big) &\Le \big( e^{-\frac{\tau}{2}\Lel} u(t_n) \big)^2 - \varepsilon^2\Le e^{-\frac{\tau}{2}\Lel} \big(e^{-\frac{\tau}{2}\Lel} u(t_n)\big)\Le u^2(t_n) \|_r\\
&= \|\varepsilon^2 e^{-\frac{\tau}{2}\Lel} \Le \big(e^{-\frac{\tau}{2}\Lel} u(t_n)\big) \Le \big( \big( e^{-\frac{\tau}{2}\Lel} u(t_n) \big)^2 -u^2(t_n)\big) \|_r \\
&\leq \|\varepsilon  \big(e^{-\frac{\tau}{2}\Lel} u(t_n)\big) \Le \big( \big( e^{-\frac{\tau}{2}\Lel} u(t_n) \big)^2 -u^2(t_n)\big) \|_{r-1}\\
&\leq \tau  C_r^2 \varepsilon^{1-\sigma} \|u(t_n)\|_{r+\lambda-2\sigma +1},
\end{aligned} 
\end{equation}
where the first and third inequalities follow by Lemmas \ref{lemma_e-1bound} and \ref{lemma_bound_eLe}, the second one by \eqref{bil_est} and Lemma \ref{lemma_lin_iso} and the last inequality follows by Lemma  \ref{lemma_bound_eLe}. Thus, by \eqref{strang_err1}, \eqref{strang_err2} and \eqref{strang_err3}, we can conclude that
\begin{gather}\label{locStrang}
\| \Phi^{\tau}(u(t_n)) - \hat{\Psi}^{\tau}(u(t_n))\|_r \leq \tau^3 \varepsilon^{1-\sigma} K_2,
\end{gather}
for some $K_2=K_2(\|u\|_{r+2\lambda-2\sigma +1})>0$.

\noindent  {\bf Stability analysis.}
Using Lemma \ref{lemma_bound_eLe}, \eqref{lin_iso} and \eqref{bil_est} we obtain
\begin{equation}
\label{strang_s1}
\begin{aligned}
\|\hat{\Psi}(&f(t_n))-\hat{\Psi}(g(t_n))\|_r \leq\, \|f(t_n)-g(t_n)\|_r + \| \tau \varepsilon\Le\big( e^{-\frac{\tau}{2}\Lel} f(t_n) \big)^2 - \tau \varepsilon\Le\big( e^{-\frac{\tau}{2}\Lel} g(t_n) \big)^2\|_r\\
&+\| \tau^2\varepsilon \big(e^{-\frac{\tau}{2}\Lel} f(t_n)\big) \Le \big( e^{-\frac{\tau}{2}\Lel} f(t_n) \big)^2 - \tau^2\varepsilon \big(e^{-\frac{\tau}{2}\Lel} g(t_n)\big) \Le \big( e^{-\frac{\tau}{2}\Lel} g(t_n) \big)^2  \|_r. 
\end{aligned}
\end{equation}
In particular,
\begin{gather}
\label{strang_s2}
\| \tau \big( e^{-\frac{\tau}{2}\Lel} f(t_n) \big)^2 - \tau \big( e^{-\frac{\tau}{2}\Lel} g(t_n) \big)^2\|_r\leq 
 \tau C_r \| f(t_n)  + g(t_n)\|_r \| f(t_n)-   g(t_n)\|_r
\end{gather}
and, similarly, 
\begin{equation}
\label{strang_s3}
\begin{aligned}
\| \tau^2\varepsilon \big(e^{-\frac{\tau}{2}\Lel} f(t_n)\big) \Le \big( e^{-\frac{\tau}{2}\Lel} f(t_n) \big)^2 - \tau^2\varepsilon \big(e^{-\frac{\tau}{2}\Lel} g(t_n)\big) \Le \big( e^{-\frac{\tau}{2}\Lel} g(t_n) \big)^2  \|_r \\
\leq \tau^2 C_r^2\|f(t_n)-g(t_n)\|_r\|f(t_n)\|^2_{r-1} + \tau^2 C_r^2 \|f(t_n)\|_r\| f(t_n)+g(t_n) \|_{r-1}\| f(t_n)-g(t_n) \|_{r-1}.
\end{aligned}
\end{equation}
Collecting our results in \eqref{strang_s1}, \eqref{strang_s2} and \eqref{strang_s3} leads to the following stability estimate
\begin{gather}\label{stabStrang}
\| \hat{\Psi}^{\tau}(f(t_n))-\hat{\Psi}^{\tau}(g(t_n)) \|_r \leq (1+\tau M_2)\| f(t_n)-   g(t_n)\|_r,
\end{gather}
for some $M_2=M_2(\|f\|_r,\|g\|_r,\tau)>0$. The local error estimate \eqref{locStrang} together with the stability bound \eqref{stabStrang} allow us to 
 conclude by a Lady Windermere fan argument  (see \cite{hairer_geometric_2006}).
\end{proof}
\section{A third order Splitting method for BBM}
\subsection{Derivation}
In this section we present a third order splitting scheme which is derived in \cite{ruth_canonical_1983} for the integration of Hamilton's equations and takes the form
\begin{gather}
\label{MET3}
\tilde{\Psi}^{\tau} = \Psi_2^{d_3\tau}\circ\tilde{\Psi}_1^{c_3\tau}\circ\Psi_2^{d_2\tau}\circ\tilde{\Psi}_1^{c_2\tau}\circ\Psi_2^{d_1\tau}\circ\tilde{\Psi}_1^{c_1\tau},
\end{gather}
with the weights
\begin{gather}
\label{weights_met3}
c_1=\frac{7}{24},\, c_2=\frac{3}{4},\, c_3=-\frac{1}{24},\quad d_1=\frac{2}{3},\, d_2=-\frac{2}{3}\,\,\text{and}\,\, d_3=1.
\end{gather}
The linear flow $\Psi_2^t(\cdot)$ is again given by \eqref{Psi2} and we are left with deriving a suitable third order approximation $\tilde{\Psi}_1^t(\cdot)$ to the the nonlinear problem \eqref{S1}.  Taylor series expansion yields that
\begin{align}
\label{duh3}
w(t_n+\tau)=\,&w(t_n)-\varepsilon\Le\int_0^{\tau} w^2(t_n+s)\,\mathrm{d}s\notag\\
=\,&w(t_n)-\varepsilon\Le\int_0^{\tau} \bigg( w^2(t_n) + 2sw(t_n)\partial_tw(t_n) + \frac{s^2}{2} \partial_{t}^2 w^2(t_n) \bigg) \,\mathrm{d}s + R_3(w),
\end{align}
where $R_3(w)$ satisfies the following bound
\begin{align}
\label{R_3}
\|R_3(w)\|_r=\bigg\|\varepsilon\Le \int_0^{\tau} \frac{s^3}{6}\,\mathrm{d}s \partial_{t}^3w^2 (t)_{t=\xi} \bigg\|_r \leq \varepsilon^{1-\sigma}\tau^4 K,
\end{align}
for some $K=K(\sup_{t_n \leq t \leq t_{n+1}}\|w(t)\|_{r+1-2\sigma })>0$, by Lemma \ref{bound_partialw}. 

The expansion \eqref{duh3} motivates us to define the numerical nonlinear flow as follows
\begin{equation}
\label{psi1_met3}
\begin{aligned}
\tilde{\Psi}_1^{\tau}(w^n):=\,&w^n- \tau \varepsilon\Le w^n + \tau^2\varepsilon^2 \Le w^n\Le (w^n)^2 - \frac{\tau^3}{6}\varepsilon^3\Le\big(\Le (w^n)^2\big)^2\\
 &-\frac{2}{3}\tau^3\varepsilon^3\Le w^n\Le w^n\Le (w^n)^2
\end{aligned}
\end{equation}
which together with \eqref{MET3} defines our third order splitting scheme.
\subsection{Error Analysis of the third order splitting scheme}
\label{errorMET3}


\begin{theorem}
Fix $0\leq \sigma \leq 1$  and $r\geq 0$ such that $r+1+3\lambda-2\sigma>1/2$  and assume that the exact solution of \eqref{BBM}  satisfies $u \in \mathcal{C}\left([0,T],H^{r+1+3\lambda-2\sigma}(\mathbb{T})\right)$. Then there exists a  $\tau_0>0$ such that for $\tau\leq\tau_0$ and $t_n\leq T$, the third order splitting scheme $u^n$ defined in \eqref{MET3} allows the global error estimate
\begin{gather*}
\|u(t_n)-u^n\|_r \leq \varepsilon^{1-\sigma} \tau^3 K, 
\end{gather*}
for a constant $K=K(\sup_{t\in[0,T]}\|u(t)\|_{r+1+3\lambda-2\sigma})>0$.
\end{theorem}

\begin{proof} We use the main tools in the convergence analysis of high order splitting schemes, namely their commutator structure. As the splitting scheme now has three stages, a naive approach analogous to the previous sections becomes much more involved. Therefore, we will exploit the general local error structure of high order splitting methods, see, e.g., \cite{descombes_exact_2010,hairer_geometric_2006}. In our local error analysis it will be in  particular important  to study the gain/loss in regularity and $\varepsilon$. 
For this purpose we will express the exact flow of the nonlinear subproblem \eqref{S1} in terms of its Lie derivative. For an introduction to Lie derivatives, see, e.g.,  \cite{hairer_geometric_2006}.   In case of our   nonlinear subproblem \eqref{S1} the Lie derivative takes the form
\begin{gather*}
D_1=\varepsilon\Le w^2\frac{\partial}{\partial_w}
\end{gather*}
and allows us to express the exact flow $\Phi_1^t(\cdot)$ of  \eqref{S1} by
\begin{gather}
\label{exptD2}
\Phi_1^t(w)=e^{tD_1}w,\quad w\in H^r.
\end{gather}
For details we again refer to \cite{hairer_geometric_2006}. Note that, similarly to  \eqref{lin_iso}, the  operator flow $e^{tD_1}$  is bounded on 
 $H^{r}(\mathbb{T})$, as the Lie derivative $D_1$ is bounded  uniformly in $\varepsilon$ by Lemma \ref{lemma_bound_eLe}. 

\noindent{\bf Local Error analysis.} We add and subtract the splitting scheme given by the weights $c_i,\,d_i,$ $i=1,2,3$ (cf \eqref{weights_met3}), and the exact flows $\Phi_1^t(\cdot)$ and $\Phi_2^t(\cdot)$. This way we find an upper bound for the local error by means of the commutator error and the error $R_3$ given in \eqref{R_3}. More precisely, we obtain 
\begin{equation}
\label{local_error_met3}
\begin{aligned}
\|\Phi^{\tau}(u(t_n)) - \tilde{\Psi}^{\tau}(u(t_n))\|_r \leq\,& \| \Phi^{\tau}(u(t_n)) - \big(\Phi_2^{d_3\tau}\circ\Phi_1^{c_3\tau}\circ\Phi_2^{d_2\tau}\circ\Phi_1^{c_2\tau}\circ\Phi_2^{d_1\tau}\circ\Phi_1^{c_1\tau}\big)(u(t_n))  \|_r\\
 &+ \| \big(\Phi_2^{d_3\tau}\circ\Phi_1^{c_3\tau}\circ\Phi_2^{d_2\tau}\circ\Phi_1^{c_2\tau}\circ\Phi_2^{d_1\tau}\circ\Phi_1^{c_1\tau}\big)(u(t_n)) - \tilde{\Psi}^{\tau}(u(t_n)) \|_r.
\end{aligned}
\end{equation}
Thanks to \cite[Theorem 1]{descombes_exact_2010}, the first term can be expressed as follows
\begin{gather*}
\| \Phi^{\tau}(u(t_n)) - \big(\Phi_2^{d_3\tau}\circ\Phi_1^{c_3\tau}\circ\Phi_2^{d_2\tau}\circ\Phi_1^{c_2\tau}\circ\Phi_2^{d_1\tau}\circ\Phi_1^{c_1\tau}\big)(u(t_n)) \|_r = \bigg\|  \bigg( \int_0^{\tau} \Phi^{\tau-s}\mathcal{R}(s) \,\mathrm{d}s \bigg)(u(t_n)) \bigg\|_r,
\end{gather*}
where the remainder has the form
\begin{align*}
\mathcal{R}(s)= \Phi_2^{d_3\tau}\circ\Phi_1^{c_3\tau}\circ \Omega(s) \circ\Phi_2^{d_2\tau}\circ\Phi_1^{c_2\tau}\circ\Phi_2^{d_1\tau}\circ\Phi_1^{c_1\tau},
\end{align*} 
and, by the choice of $c_i,\,d_i,$ $i=1,2,3$, (which satisfy the order conditions (4.3a) and (4.3b) in \cite{descombes_exact_2010}),  $\Omega$ has the form
\begin{align*}
\Omega(s) = c_1 \mathcal{J}_+^s(d_2D_1,\mathcal{J}_+^{\tilde{s}}(c_2\Lel,\mathcal{J}_+^{\tilde{\tilde{s}}}(d_1D_1,\Lel)))
\end{align*}
with
\begin{align}
\label{J+-}
\mathcal{J}_{\pm}^s(L_1,L_2)=\int_0^s e^{\pm tL_1}[L_1,L_2]e^{\mp t L_1}\,\mathrm{d}t, 
\end{align}
for some linear operators. Note that therefore we have 
four nested commutators in the remainder with two factors of $D_1$ and $\Lel$ respectively. Note that we loose regularity under the action of $\Lel$, by Lemma \ref{lemma_bound_Le}, whereas the action of $D_1$ allows us to gain  in regularity, respectively, powers of $\varepsilon$, see Lemma \ref{lemma_bound_eLe}. This allows a bound of the form
\begin{gather}
\label{K31}
\| \Phi^{\tau}(u(t_n)) -\big(\Phi_2^{d_3\tau}\circ\Phi_1^{c_3\tau}\circ\Phi_2^{d_2\tau}\circ\Phi_1^{c_2\tau}\circ\Phi_2^{d_1\tau}\circ\Phi_1^{c_1\tau}\big)(u(t_n)) \|_r \leq\tau^4 \varepsilon^{1-\sigma}  K_{3,1},
\end{gather}
for some $K_{3,1}=K_{3,1}(\|u(t_n)\|_{r+1+3\lambda-2\sigma})>0$. On the other hand, we need to derive a bound on  the second term in \eqref{local_error_met3}, which arises from the approximations of the exact nonlinear flow $\tilde{\Psi}_1^t(\cdot)$ carried out in \eqref{MET3}. Adding and subtracting the terms
\begin{align*}
\big(\Psi_2^{d_3\tau}\circ\Phi_1^{c_3\tau}\circ\Psi_2^{d_2\tau}\circ\tilde{\Psi}_1^{c_2\tau}\circ\Psi_2^{d_1\tau}\circ\tilde{\Psi}_1^{c_1\tau}\big)(u(t_n))\,\, \text{and}\,\,\big(\Psi_2^{d_3\tau}\circ\Phi_1^{c_3\tau}\circ\Psi_2^{d_2\tau}\circ\Phi_1^{c_2\tau}\circ\Psi_2^{d_1\tau}\circ\tilde{\Psi}_1^{c_1\tau}\big)(u(t_n))
\end{align*}
gives
\begin{gather*}
\| \big(\Phi_2^{d_3\tau}\circ\Phi_1^{c_3\tau}\circ\Phi_2^{d_2\tau}\circ\Phi_1^{c_2\tau}\circ\Phi_2^{d_1\tau}\circ\Phi_1^{c_1\tau}\big)(u(t_n)) - \tilde{\Psi}^{\tau}(u(t_n)) \|_r \\
= \| \big(\Phi_2^{d_3\tau}\circ\Phi_1^{c_3\tau}\circ\Phi_2^{d_2\tau}\circ\Phi_1^{c_2\tau}\circ\Phi_2^{d_1\tau}\circ\Phi_1^{c_1\tau}\big)(u(t_n)) - \big( \Psi_2^{d_3\tau}\circ\tilde{\Psi}_1^{c_3\tau}\circ\Psi_2^{d_2\tau}\circ\tilde{\Psi}_1^{c_2\tau}\circ\Psi_2^{d_1\tau}\circ\tilde{\Psi}_1^{c_1\tau} \big) (u(t_n)) \|_r \\
\leq \| \big(\Phi_2^{d_3\tau}\circ\Phi_1^{c_3\tau}\circ\Phi_2^{d_2\tau}\circ\Phi_1^{c_2\tau}\circ\Phi_2^{d_1\tau}\circ\Phi_1^{c_1\tau}\big)(u(t_n))) - \big(\Psi_2^{d_3\tau}\circ\Phi_1^{c_3\tau}\circ\Psi_2^{d_2\tau}\circ\Phi_1^{c_2\tau}\circ\Psi_2^{d_1\tau}\circ\tilde{\Psi}_1^{c_1\tau}\big)(u(t_n)) \|_r\\
+\| \big(\Psi_2^{d_3\tau}\circ\Phi_1^{c_3\tau}\circ\Psi_2^{d_2\tau}\circ\Phi_1^{c_2\tau}\circ\Psi_2^{d_1\tau}\circ\tilde{\Psi}_1^{c_1\tau}\big)(u(t_n)) - \big(\Psi_2^{d_3\tau}\circ\Phi_1^{c_3\tau}\circ\Psi_2^{d_2\tau}\circ\tilde{\Psi}_1^{c_2\tau}\circ\Psi_2^{d_1\tau}\circ\tilde{\Psi}_1^{c_1\tau}\big)(u(t_n)) \|_r\\
+\| \big(\Psi_2^{d_3\tau}\circ\Phi_1^{c_3\tau}\circ\Psi_2^{d_2\tau}\circ\tilde{\Psi}_1^{c_2\tau}\circ\Psi_2^{d_1\tau}\circ\tilde{\Psi}_1^{c_1\tau}\big)(u(t_n)) - \big(\Psi_2^{d_3\tau}\circ\tilde{\Psi}_1^{c_3\tau}\circ\Psi_2^{d_2\tau}\circ\tilde{\Psi}_1^{c_2\tau}\circ\Psi_2^{d_1\tau}\circ\tilde{\Psi}_1^{c_1\tau}\big)(u(t_n)) \|_r \\
\leq K' \| \big(\Phi_1^{c_3\tau}\circ\Phi_2^{d_2\tau}\circ\Phi_1^{c_2\tau}\circ\Phi_2^{d_1\tau}\big)\big(R_3(u(t_n))\big)\|_r
+ K' \| \big(\Phi_1^{c_3\tau}\circ\Phi_2^{d_2\tau}\big)\big(R_3(\Phi_1^{d_1\tau}\circ\tilde{\Psi}_1^{c_1\tau}(u(t_n))\big)\|_r\\
+K' \| R_3(\Phi_1^{d_2\tau}\circ\tilde{\Psi}_1^{c_2\tau}\circ\Phi_1^{d_1\tau}\circ\tilde{\Psi}_1^{c_1\tau}(u(t_n))\|_r,
\end{gather*}
 for some $K'>0$ chosen such that $\| e^{\tau D_1}g\|_r\leq K' \|g\|_r$ holds for all $g\in H^r$. In addition, the last inequality follows by definition of $R_3$, see \eqref{R_3}. Thanks to \eqref{bil_est} and Lemma~\ref{lemma_bound_eLe} we can bound  $\tilde{\Psi}_1^{t}(\cdot)$ defined in \eqref{psi1_met3} for all $t\in \mathbb{R}$ as follows
\begin{align*}
\| \tilde{\Psi}_1^{t}f\|_r &\leq \|f\|_r+\tau\|f\|_{r-1}+\tau^2C_r\|f\|_{r-1}\|f^2\|_{r-1} + \frac{\tau^3}{6}C_r\|f^2\|_{r-2}^2 + \frac{2}{3}\tau^3C_r^2\|f\|_{r-1}\|f\|_{r-2}\|f\|_{r-3}\\
&\leq K,
\end{align*}
for some $K=K(\|f\|_{r})>0$.  We thus obtain the following bound, using \eqref{R_3} and Lemma \ref{lemma_lin_iso},
\begin{gather}
\label{K32}
\| \big(\Phi_2^{c_1\tau}\circ\Phi_1^{d_1\tau}\circ\Phi_2^{c_2\tau}\circ\Phi_1^{d_2\tau}\circ\Phi_2^{c_3\tau}\circ\Phi_1^{d_3\tau}\big)(u(t_n)) - \tilde{\Psi}^{\tau}(u(t_n)) \|_r \leq \tau^4 \varepsilon^{1-\sigma}K_{3,2},
\end{gather}
for some $K_{3,2}=K_{3,2}(\|u(t_n)\|_{r+1+3\lambda -2\sigma})>0$. Finally, plugging \eqref{K31} and \eqref{K32} into \eqref{local_error_met3} we obtain the total local error bound
\begin{align}\label{loc33}
\|\Phi^{\tau}(u(t_n)) - \tilde{\Psi}^{\tau}(u(t_n))\|_r \leq \tau^3 \varepsilon^{1-\sigma} (K_{3,1}+K_{3,2}).
\end{align}

\noindent{\bf Stability analysis.} We first prove a stability estimate for $\tilde{\Psi}_1^t(\cdot)$. For $i\in\{1,2,3 \}$ we have
\begin{align*}
\|\tilde{\Psi}_1^{c_i\tau}f(t_n)-\tilde{\Psi}_1^{c_i\tau}g(t_n)\|_r \leq\,&\|   f(t_n)-g(t_n)\|_r + \tau \|\varepsilon\Le (f(t_n)-g(t_n))\|_r\\
&+\tau^2\|\varepsilon^2\Le f(t_n)\Le f(t_n)^2 - \varepsilon^2\Le g(t_n)\Le g(t_n)^2\|_r\\
&+\frac{\tau^3}{6}\|\varepsilon^2\Le \big(\Le f(t_n)^2\big)^2-\varepsilon^3\Le \big(\Le g(t_n)^2\big)^2\|_r\\
&+\frac{2}{3}\tau^3\| \varepsilon^2\Le f(t_n) \Le f(t_n)^2- \varepsilon^2\Le g(t_n) \Le g(t_n)^2\|_r\\
\leq\,& (1+\tau M_3')\|f(t_n)-g(t_n)\|_r,
\end{align*}
for some $M_3'=M_3'(\|f\|_r,\|g\|_r)>0$. Note that the last inequality follows by \eqref{bil_est} and Lemma \ref{lemma_bound_Le}. Iterating this argument one obtains, using Lemma \ref{lemma_lin_iso},
\begin{align*}
\|\tilde{\Psi}^{\tau}(f(t_n))-\tilde{\Psi}^{\tau}(g(t_n))\|_r 
&= \| \big(  \Psi_2^{d_3\tau}\circ\tilde{\Psi}_1^{c_3\tau}\circ\Psi_2^{d_2\tau}\circ\tilde{\Psi}_1^{c_2\tau}\circ\Psi_2^{d_1\tau}\circ\tilde{\Psi}_1^{c_1\tau} \big) (f(t_n)-g(t_n))\|_r\\
&\leq (1+\tau M_3') \| \big(  \Psi_2^{d_2\tau}\circ\tilde{\Psi}_1^{c_2\tau}\circ\Psi_2^{d_1\tau}\circ\tilde{\Psi}_1^{c_1\tau} \big) (f(t_n)-g(t_n)) \|_r \\
&\leq (1+\tau M_3')^2 \| \big(  \Psi_2^{d_1\tau}\circ\tilde{\Psi}_1^{c_1\tau} \big) (f(t_n)-g(t_n)) \|_r \\
&\leq (1+\tau M_3')^3 \|f(t_n)-g(t_n)\|_r \\
&\leq (1+\tau M_3) \|f(t_n)-g(t_n)\|_r,
\end{align*}
for $M_3:=3M_3' + 3\tau M_3'^2 + \tau^2M_3'^3>0$.

The above stability estimate together with the local error estimate \eqref{loc33} allows us to conclude by a Lady Windermere's fan argument  (see \cite{hairer_geometric_2006}).
\end{proof}

\section{A fourth order Splitting method for BBM}\label{sec:four}
\subsection{Derivation}
In this section we present a fourth order splitting scheme given by
\begin{gather}
\label{MET4}
\check{\Psi}^{\tau} = \Psi^{d_4\tau}_2\circ\check{\Psi}^{c_4\tau}_1\circ\Psi^{d_3\tau}_2\circ\check{\Psi}^{c_3\tau}_1\circ\Psi^{d_2\tau}_2\circ\check{\Psi}^{c_2\tau}_1\circ\Psi^{d_1\tau}_2\circ\check{\Psi}^{c_1\tau}_1,
\end{gather}
with
\begin{gather*}
c_1=0,\quad c_2=c_4=\sigma_1=\frac{1}{2-\sqrt[3]{2}},\quad c_3=\sigma_2=\frac{-\sqrt[3]{2}}{2-\sqrt[3]{2}},\\
d_1=d_4=\frac{1}{2}\sigma_1,\quad d_2=d_3=\frac{1}{2}(\sigma_1+\sigma_2).
\end{gather*}
For details on the order condition, we refer to   \cite{hairer_geometric_2006}. 

Again it remains to construct a  fourth order integrator for the nonlinear supbroblem \eqref{S1}. For this purpose we expand the corresponding  mild formulation up to order four
\begin{align}
\label{duh4}
w(t_n+\tau)=\,&w(t_n)-\varepsilon\Le\int_0^{\tau} w^2(t_n+s)\,\mathrm{d}s\notag\\
=\,&w(t_n)-\varepsilon\Le\int_0^{\tau} \bigg( w^2(t_n) + 2sw(t_n)\partial_tw(t_n) + \frac{s^2}{2} \partial_{t}^2 w^2(t_n) +\frac{s^3}{6}\partial_{t}^3 w^2(t_n) \bigg) \,\mathrm{d}s + R_4(w),
\end{align}
where $R_4(w)$, for some $\xi\in [t_n,t_{n+1}]$, satisfies the following bound
\begin{align}
\label{R_4}
\|R_4(w)\|_r=\bigg\| \varepsilon\Le \int_0^{\tau} \frac{s^4}{24}\,\mathrm{d}s\,\, \partial_{t}^4 w^2 (t)_{t=\xi} \bigg\|_r \leq \varepsilon^{1-\sigma}\tau^5 K,
\end{align}
for some $K=K(\sup_{t_n \leq t \leq t_{n+1}}\|w(t)\|_{r+1-2\sigma })>0$, by Lemma \ref{bound_partialw}. 

The above expansion motivates us to define the following numerical flow \begin{equation}
\begin{aligned}
\check{\Psi}_1^{\tau}(w^n) :=\,&w^n- \tau \varepsilon\Le w^n + \tau^2\varepsilon^2 \Le w^n\Le (w^n)^2 - \frac{1}{6}\tau^3\varepsilon^3\Le\big(\Le (w^n)^2\big)^2 \\ &-\frac{2}{3}\tau^3\varepsilon^3\Le w^n\Le w^n\Le (w^n)^2 
+\frac{1}{2}\tau^4 \varepsilon^4 \Le \big(\Le (w^n)^2  \big)\big(\Le w^n\Le (w^n)^2  \big)\\ &+ \frac{1}{6} \tau^4 \varepsilon^4 \Le w^n\Le \big(\Le (w^n)^2  \big)^2
+ \frac{1}{3}\tau^4\varepsilon^4 \Le w^n \Le w^n \Le w^n \Le (w^n)^2.
\end{aligned}\label{4n}
\end{equation}
Similarly to the convergence analysis of the third order splitting scheme of Section \ref{errorMET3} we obtain the following fourth order convergence result.
\begin{theorem}
Fix $0\leq \sigma \leq 1$  and $r\geq 0$ such that $r+1+3\lambda-2\sigma>1/2$  and assume that the exact solution of \eqref{BBM}  satisfies $u \in \mathcal{C}\left([0,T],H^{r+1+2\lambda-2\sigma}(\mathbb{T})\right)$. Then there exists a  $\tau_0>0$ such that for $\tau\leq\tau_0$ and $t_n\leq T$ the fourth order splitting $u^n$ defined in \eqref{MET4} (together with \eqref{4n}) satisfies the global error estimate
\begin{gather*}
\|u(t_n)-u^n\|_r \leq \varepsilon^{1-\sigma} \tau^4 K, 
\end{gather*}
for a constant $K=K(\sup_{t\in[0,T]}\|u(t)\|_{r+1+4\lambda -2\sigma})>0$. 
\end{theorem}
\begin{proof}
The proof  follows with similar arguments as given in Section \ref{errorMET3} for the third order method and will be omitted here.
\end{proof}

\section{Numerical Experiments}\label{sec:num}
In this section we underline our theoretical convergence result \eqref{result} with numerical experiments. In particular we observe that our splitting schemes convergence with desired order $\mathcal{O}(\tau^p \varepsilon)$ for $p=1,2,3,4$. For the spatial discretisation we employ  a standard Fourier pseudospectral method. More specifically, we choose the highest Fourier mode to be $M=200$, which corresponds to $\Delta x \approx 0.0314$ and integrate the following initial value up to time $T=5$
\begin{gather*}
u(0,x)= \frac{3 \sin(2x)}{2-\cos(x)}.
\end{gather*}

In Figure \ref{EOC_plot} we plot the time-step size versus the discrete $L^2$ error of the first~\eqref{MET1}, second \eqref{MET2}, third \eqref{MET3} and fourth \eqref{MET4} order splitting scheme  for different values of $\varepsilon$. In order to generate this result we used the method itself with a step size of $\tau = 10^{-15}$ as a reference solution, after comparing its accuracy with an approximation via the ode45 solver, integrated in Matlab with a very fine time step size.

 \begin{figure}[h!]
  \includegraphics[width=0.65\textwidth]{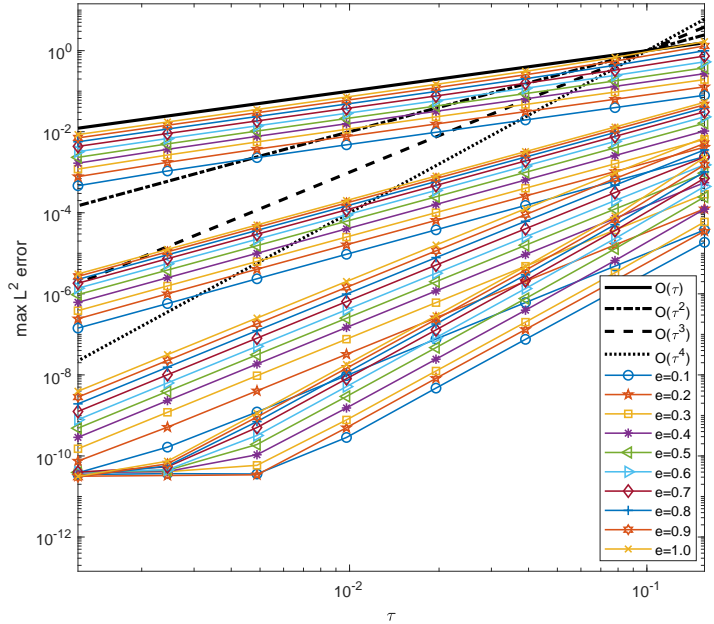}
  \caption{Order plots of the Lie splitting \eqref{MET1}, the Strang splitting \eqref{MET2}, a third order splitting \eqref{MET3} and fourth order splitting \eqref{MET4} for different values of $\varepsilon = 0.1,0.2,0.3,0.4,0.5,0.6,0.7,0.8,0.9,1$. The numerical experiment underlines the theoretical observed convergence  rate $\tau^p \varepsilon$.}
  \label{EOC_plot}
\end{figure}

In Figure \ref{ACKdV_plot} we furthermore underline the asymptotic convergence of our splitting method to the  the   KdV limit equation \eqref{kdvapp}. For this purpose we plot the difference between the Lie splitting solution of the BBM  equation \eqref{BBM} and the  numerical solution of the  KdV equation \eqref{kdvapp} for different values of $\varepsilon$. The numerical experiment underlines the asymptotic convergence of order $\mathcal{O}(\varepsilon)$.
 
 \begin{figure}[h!]
  \includegraphics[width=0.5\textwidth]{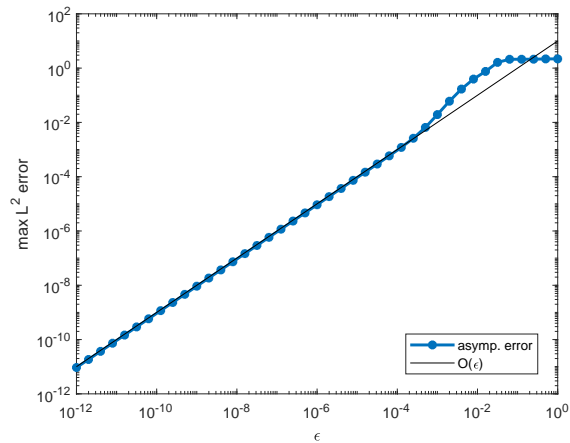}
  \caption{Asymptotic convergence of   BBM to KdV for the Lie splitting method. The numerical experiment underlines the asymptotic convergence  rate $\varepsilon$}
  \label{ACKdV_plot}
\end{figure}

\subsection*{Acknowledgements}

{
This project has received funding from the European Research Council (ERC) under the European Union’s Horizon 2020 research and innovation programme (grant agreement No. 850941).
}

\nocite{*}
\bibliographystyle{abbrv}
\bibliography{BBM_final.bib}

\end{document}